\pdfoutput=1
\RequirePackage{ifpdf}
\ifpdf 
\documentclass[pdftex]{sigma}
\else
\documentclass{sigma}
\fi

\numberwithin{equation}{section}

\newtheorem{Theorem}{Theorem}[section]
\newtheorem{Corollary}[Theorem]{Corollary}
\newtheorem{Lemma}[Theorem]{Lemma}
\newtheorem{Proposition}[Theorem]{Proposition}
 { \theoremstyle{definition}
\newtheorem{Definition}[Theorem]{Definition} }

\begin{document}

\allowdisplaybreaks

\newcommand{\arXivNumber}{1612.06996}

\renewcommand{\PaperNumber}{055}

\FirstPageHeading

\ShortArticleName{Global Existence of Bi-Hamiltonian Structures on Orientable Three-Dimensional Manifolds}

\ArticleName{Global Existence of Bi-Hamiltonian Structures\\ on Orientable Three-Dimensional Manifolds}

\Author{Melike I\c{S}\.{I}M EFE and Ender ABADO\u{G}LU}
\AuthorNameForHeading{M.~I\c{s}im Efe and E.~Abado\u{g}lu}
\Address{Yeditepe University, Mathematics Department, \.{I}n\.{o}n\.{u} Mah. Kay{\i}\c{s}da\u{g}{\i} Cad.~326A,\\
 26 A\u{g}ustos Yerle\c{s}imi, 34755 Ata\c{s}ehir \.{I}stanbul, Turkey}
\Email{\href{mailto:melike.efe@yeditepe.edu.tr}{melike.efe@yeditepe.edu.tr}, \href{mailto:eabadoglu@yeditepe.edu.tr}{eabadoglu@yeditepe.edu.tr}}

\ArticleDates{Received December 21, 2016, in f\/inal form July 04, 2017; Published online July 14, 2017}

\Abstract{In this work, we show that an autonomous dynamical system def\/ined by a nonvanishing vector f\/ield on an orientable three-dimensional manifold is globally bi-Hamiltonian if and only if the f\/irst Chern class of the normal bundle of the given vector f\/ield vanishes. Furthermore, the bi-Hamiltonian structure is globally compatible if and only if the Bott class of the complex codimension one foliation def\/ined by the given vector f\/ield vanishes.}


\Classification{53D17; 53D35}

\rightline{\textit{Dedicated to the memory of Ali Yavuz.}}

\section{Introduction}

An autonomous dynamical system on a manifold $M$
\begin{gather}
\dot{x}(t) =v(x(t))\label{e1}
\end{gather}
is determined by a vector f\/ield $v(x) $ on a manifold up to time reparametrization. Important geometric quantities related to a dynamical system are functions~$I$ which are invariant under the f\/low of the vector f\/ield
\begin{gather*}
\mathcal{L}_{v}I=0.
\end{gather*}
It is sometimes possible to relate the vector f\/ield to an invariant function via a Poisson structure~$\mathcal{J},$ which is a bivector f\/ield on~$M$
\begin{gather*}
\mathcal{J}\colon \ \Lambda ^{1}(M)\rightarrow \mathfrak{X}(M)
\end{gather*}
satisfying the Jacobi identity condition
\begin{gather*}
 [ \mathcal{J},\mathcal{J} ] _{\rm SN}=0,
\end{gather*}
where $[\, ,\,] _{\rm SN}$ is the Schouten--Nijenhuis bracket. The local structure of such manifolds was f\/irst introduced in~\cite{Wein}. The invariants satisfying the condition
\begin{gather}
v=\mathcal{J}({\rm d}I) \label{e5}
\end{gather}
are called Hamiltonian functions of (\ref{e1}). Given a~dynamical system on $M$ def\/ined by the vector f\/ield $v$, the vector f\/ield~$v$ is called a~Hamiltonian vector f\/ield if there exists a Poisson bivector~$\mathcal{J}$ and a smooth function $I$ such that equation~(\ref{e5}) holds.

\looseness=-1 Given a vector f\/ield $v$ on $M$, f\/inding a Poisson structure according to which the vector f\/ield becomes Hamiltonian may not be an easy task in general. However, if a given dynamical system can be put into Hamiltonian form then, there may be more than one Poisson structure which makes it into a Hamiltonian system. In~\cite{Magri}, a bi-Hamiltonian system is introduced for the analysis of certain inf\/inite-dimensional soliton equations. In such a case, there arises the question of the relation between these Poisson structures, which is called compatibility. Although there are at least two dif\/ferent approaches to compatibility~\cite{Santo}, by following~\cite{Olver} we adapt the def\/initions below:

\begin{Definition}\label{definition1}
A dynamical system is called bi-Hamiltonian if it can be written in Hamiltonian form in two distinct ways:
\begin{gather}
v=\mathcal{J}_{1}( {\rm d}H_{2}) =\mathcal{J}_{2}( {\rm d}H_{1}),\label{e6}
\end{gather}
such that $\mathcal{J}_{1}$ and $\mathcal{J}_{2}$ are nowhere multiples of each other. This bi-Hamiltonian structure is compatible if $\mathcal{J}_{1}+\mathcal{J}_{2}$ is also a Poisson structure.
\end{Definition}

In this paper we conf\/ine ourselves to dynamical systems on three-dimensional orientable manifolds. For three-dimensional manifolds, where there is no symplectic structure for dimensional reasons, Poisson structures have a simple form. Poisson structures of dynamical systems on three manifolds are extensively studied f\/irst in~\cite{Nutku} and then also in~\cite{Haas} and \cite{Gengoux}. Following the def\/initions in \cite{Nutku}, choosing any Riemannian metric~$\boldsymbol{g}$ on~$M$, a Poisson bivector f\/ield, which is skew-symmetric, can be associated to a vector f\/ield by using the Lie algebra isomorphism $\mathfrak{so}(3) \simeq \mathbb{R}^{3}$
\begin{gather*}
\mathcal{J}=\mathcal{J}^{mn}e_{m}\wedge e_{n}=\varepsilon _{k}^{mn}J^{k}e_{m}\wedge e_{n},
\end{gather*}
and the vector f\/ield
\begin{gather*}
J=J^{k}e_{k}
\end{gather*}
is called the Poisson vector f\/ield on~$M$.

Then, the Jacobi identity has the form
\begin{gather}
J\cdot ( \nabla \times J) =0, \label{e9}
\end{gather}
and equation (\ref{e6}) becomes
\begin{gather}
v=J_{1}\times \nabla H_{2}=J_{2}\times \nabla H_{1}. \label{e10}
\end{gather}
Since $J_{1}$ and $J_{2}$ are not multiples of each other by def\/inition, we have
\begin{gather}
J_{1}\times J_{2}\neq 0 \label{e11}
\end{gather}
and
\begin{gather}
J_{i}\cdot v=0 \label{e12}
\end{gather}
for $i=1,2$.

This work is focused on the bi-Hamiltonian structure of dynamical systems def\/ined by non\-va\-nishing vector f\/ields on orientable three-dimensional manifolds, or equivalently vector f\/ields on three-dimensional manifolds whose supports are orientable three-dimensional manifolds. Since all orientable three-dimensional manifolds are parallelizable~\cite{Steifel}, there is no topological obstruction to the global existence of a nonvanishing vector f\/ield. Then, by the bi-Hamiltonian form~(\ref{e10})--(\ref{e12}), $\{ v,J_{1},J_{2}\} $ forms a local frame f\/ield. Therefore, whenever the system is globally bi-Hamiltonian, $\{v,J_{1},J_{2}\} $ becomes a global frame f\/ield on~$M$. For example, for $M=\mathbb{R}^{3}$ and $v=\partial _{x^{0}}$ we have $J_{i}=\partial _{x^{i}}$ and $\{ \partial _{x^{0}},\partial _{x^{1}},\partial _{x^{2}}\} $ forms such a global frame f\/ield. However, the global existence of the frame f\/ield $\{ v,J_{1},J_{2}\} $ is by no means guaranteed. The simplest counterexample is the gradient f\/low of the $S^{2}$ in $\mathbb{R}^{3}\setminus \{0\} $. Here, a frame f\/ield $\{ v,J_{1},J_{2}\}$ cannot be def\/ined globally since $J_{1}$, $J_{2}$ are sections of the tangent bundle of $S^{2}$ which is not trivial and does not admit two nonvanishing linearly independent vector f\/ields.

The goal of this paper is to give necessary and suf\/f\/icient conditions for a~nonvanishing vector f\/ield on an orientable three-dimensional manifold to admit a compatible bi-Hamiltonian structure. The paper is organized as follows: In Section~\ref{section2}, the local existence of bi-Hamiltonian systems is investigated in a neighbourhood of a point, possibly ref\/ined by the existence conditions of solutions of certain ODE's related with the problem, and it is shown in Theorem~\ref{theorem1} that it is always possible to f\/ind a pair of compatible Poisson structures such that the system def\/ined by the nonvanishing vector f\/ield becomes bi-Hamiltonian. In Section~\ref{section3}, obstructions to the global existence of a pair of Poisson structures are studied. In Section~\ref{section3.2} the primary obstruction for the existence of a global pair of Poisson structures is investigated, and it is shown in Theorem~\ref{theorem3} that such a pair, which is not necessarily compatible, exists if and only if the f\/irst Chern class of the normal bundle vanishes. Finally, the global compatibility of this pair is investigated in Section~\ref{section3.3} and it is shown in Theorem~\ref{theorem4} that under the assumption of global existence, the vanishing of the Bott class of the complex codimension one foliation is the necessary and suf\/f\/icient condition for the global compatibility of the pair of Poisson structures.

Throughout the work, bivectors are denoted by calligraphic and forms are denoted by bold letters.

\section{Local existence of bi-Hamiltonian structure in 3D}\label{section2}

For this purpose, we will f\/irst analyze the local solutions of the equation~(\ref{e9}) def\/ining Poisson vectors, which is also studied in~\cite{Bermejo}. Let $M$ be an orientable three-dimensional manifold with an arbitrary Riemannian metric~$\boldsymbol{g}$, and~$v$ be a~nonvanishing vector f\/ield. Let
\begin{gather*}
\widehat{e}_{1}=\frac{v}{\Vert v \Vert }
\end{gather*}
and extend this vector f\/ield to a local orthonormal frame f\/ield $\{\widehat{e}_{1},\widehat{e}_{2},\widehat{e}_{3}\} $. Def\/ine the structure functions $(C_{ij}^{k}(x)) $ via the relation
\begin{gather}
[ \widehat{e}_{i},\widehat{e}_{j}] =C_{ij}^{k}(x) \widehat{e}_{k}. \label{e14}
\end{gather}

\begin{Proposition}\label{proposition1}
A nonvanishing vector field $v$ admits two independent local Poisson structures on~$M$.
\end{Proposition}

\begin{proof}
Adopting the frame def\/ined above and using (\ref{e12}), we have the Poisson vector f\/ield
\begin{gather}
J=\alpha \widehat{e}_{2}+\beta \widehat{e}_{3}, \label{e15}
\end{gather}
and its curl is given by
\begin{gather}
\nabla \times J=\nabla \alpha \times \widehat{e}_{2}+\alpha \nabla \times \widehat{e}_{2}+\nabla \beta \times \widehat{e}_{3}+\beta \nabla \times \widehat{e}_{3}. \label{e16}
\end{gather}

Now the Jacobi identity (\ref{e9}) is obtained by taking the dot product of (\ref{e15}) with (\ref{e16}), and using triple vector product identities we get
\begin{gather}
\beta \widehat{e}_{1}\cdot \nabla \alpha -\alpha \widehat{e}_{1}\cdot \nabla\beta -\alpha ^{2}C_{31}^{2}-\alpha \beta \big(C_{31}^{3}+C_{12}^{2}\big) -\beta ^{2}C_{12}^{3}=0. \label{e17}
\end{gather}

If $J=0$ then $\Vert v\Vert =0$ and hence $v=0$, which contradicts with our assumption that the vector f\/ield is nonvanishing. Therefore, we assume
\begin{gather*}
J\neq 0,
\end{gather*}
which means that $\alpha \neq 0$ or $\beta \neq 0$. Now we assume $\alpha \neq 0$, while the case $\beta \neq 0$ is similar and amounts to a~rotation of the frame f\/ields. Def\/ining
\begin{gather*}
\mu =\frac{\beta }{\alpha }
\end{gather*}
and dividing (\ref{e17}) by $\alpha ^{2}$, we get
\begin{gather}
\widehat{e}_{1}\cdot \nabla \mu =-C_{31}^{2}-\mu \big( C_{31}^{3}+C_{12}^{2}\big) -\mu ^{2}C_{12}^{3}, \label{e20}
\end{gather}
whose characteristic curve is the integral curve of (\ref{e1}) in arclength parametrization and
\begin{gather}
\frac{{\rm d}\mu }{{\rm d}s}=-C_{31}^{2}-\mu \big( C_{31}^{3}+C_{12}^{2}\big) -\mu ^{2}C_{12}^{3} \label{e21}
\end{gather}
in the arclength variable~$s$. The Riccati equation (\ref{e21}) is equivalent to a linear second order equation and hence, possesses two linearly independent solutions leading to two Poisson vector f\/ields. Since the vector f\/ield $v$ is assumed to be nonvanishing, for each $\boldsymbol{x}_{0}\in \mathbb{R}^{3}$ it is possible to f\/ind a~neighborhood foliated by the integral curves of $v$ which are nothing but characteristic curves of (\ref{e20}). Therefore, solutions of~(\ref{e21}) can be extended to a possibly smaller neighborhood on which the
Riccati equation has two independent solutions which we call $\mu _{i}$ for $i=1,2$. Hence, we have two Poisson vector f\/ields
\begin{gather}
J_{i}=\alpha _{i} \big( \widehat{e}_{2}+\mu _{i}\widehat{e}_{3} \big), \label{e22}
\end{gather}
where the coef\/f\/icients $\alpha _{i}$ are arbitrary.
\end{proof}

Note that, (\ref{e20}) determines $\mu _{i}$ alone, but not $\alpha _{i}$. Taking the advantage of the freedom of choosing arbitrary scaling factors we may restrict these factors by imposing compatibility on our Poisson vector f\/ields.

\begin{Proposition}\label{proposition2}
Two Poisson structures obtained in \eqref{e20} are compatible iff
\begin{gather}
\widehat{e}_{1}\cdot \nabla \ln \frac{\alpha _{i}}{\alpha _{j}}=C_{12}^{3} ( \mu _{i}-\mu _{j} ). \label{e23}
\end{gather}
\end{Proposition}

\begin{proof} Let
\begin{gather*}
J=J_{1}+J_{2}
\end{gather*}
Using (\ref{e9}) for $J_{1}$, $J_{2}$ and $J$
\begin{gather}
( \nabla \times J) \cdot J=( \nabla \times J_{2})\cdot J_{1}+ ( \nabla \times J_{1} ) \cdot J_{2}=0. \label{e25}
\end{gather}
For the Poisson vector f\/ields def\/ined in (\ref{e20}), taking the dot product of both sides of~(\ref{e16}) by~$J_{j}$, leads to
\begin{gather}
( \nabla \times J_{i}) \cdot J_{j}=\alpha _{i}\alpha_{j} (\mu _{i}-\mu _{j} ) \big( C_{12}^{2}+C_{12}^{3}\mu _{i}-\widehat{e}_{1}\cdot \nabla \ln \alpha _{i}\big). \label{e27}
\end{gather}
Therefore, the compatibility condition (\ref{e25}) implies that
\begin{gather*}
C_{12}^{2}+C_{12}^{3}\mu _{i}-\widehat{e}_{1}\cdot \nabla \ln \alpha_{i}=C_{12}^{2}+C_{12}^{3}\mu _{j}-\widehat{e}_{1}\cdot \nabla \ln \alpha_{j},
\end{gather*}
and hence, we get
\begin{gather}
\widehat{e}_{1}\cdot \nabla \ln \frac{\alpha _{i}}{\alpha _{j}}=C_{12}^{3}( \mu _{i}-\mu _{j}), \label{e29}
\end{gather}
whose characteristic curve is the solution curve of (\ref{e1}) in arclength parametrization
\begin{gather}
\frac{{\rm d}}{{\rm d}s}\ln \frac{\alpha _{i}}{\alpha _{j}}=C_{12}^{3} ( \mu_{i}-\mu _{j} ). \label{e30}
\end{gather}
By a similar line of reasoning as above, the solutions of (\ref{e30}) can also be extended to the whole neighborhood, and the proposition follows.
\end{proof}

However, having a pair of Poisson structures obtained in (\ref{e20}) and even a compatible pair satisfying (\ref{e29}) do not guarantee the existence of Hamiltonian functions even locally.

\begin{Proposition}\label{proposition3}
The dynamical system \eqref{e1} is locally bi-Hamiltonian with the pair of Poisson structures obtained in~\eqref{e22} if and only if
\begin{gather}
\widehat{e}_{1}\cdot \nabla \ln \frac{\alpha _{i}}{\Vert v\Vert }=C_{31}^{3}+\mu _{i}C_{12}^{3}. \label{e31}
\end{gather}
\end{Proposition}

\begin{proof}
For this purpose we f\/irst need to write down the equations for the Hamiltonian functions. The invariance condition of Hamiltonian functions under the f\/low generated by $v$ implies
\begin{gather}
\widehat{e}_{1}\cdot \nabla H_{i}=0, \label{e32}
\end{gather}
so the gradients of Hamiltonian functions are linear combinations of~$\widehat{e}_{2}$ and~$\widehat{e}_{3}$. Then, inser\-ting~(\ref{e22}) into~(\ref{e10}) we get another condition
\begin{gather}
\widehat{e}_{3}\cdot \nabla H_{j}-\mu _{i}\widehat{e}_{2}\cdot \nabla H_{j}=\frac{ \Vert v \Vert }{\alpha _{i}} \label{e35}
\end{gather}
or by def\/ining
\begin{gather*}
u_{i}=-\mu _{i}\widehat{e}_{2}+\widehat{e}_{3}
\end{gather*}
(\ref{e35}) can be written as
\begin{gather}
u_{i}\cdot \nabla H_{j}=\frac{\Vert v\Vert }{\alpha _{i}}.\label{e37}
\end{gather}
Equations (\ref{e32}) and (\ref{e37}) for Hamiltonian functions are subject to the integrability condition
\begin{gather*}
\widehat{e}_{1} ( u_{i}(H_{j}) ) -u_{i}\big( \widehat{e}_{1}(H_{j}) \big) =\big[ \widehat{e}_{1},u_{i}\big](H_{j}).
\end{gather*}
Using the commutation relations given in (\ref{e14}) and (\ref{e20}), we obtain
\begin{gather}
[ \widehat{e}_{1},u_{i}] =-\big( C_{31}^{1}+\mu_{i}C_{12}^{1}\big) \widehat{e}_{1}-\big(C_{31}^{3}+\mu_{i}C_{12}^{3}\big) u_{i}. \label{e40}
\end{gather}
Applying $H_{j}$ to both sides of (\ref{e40}) and using two equations (\ref{e32}) and (\ref{e37}) for Hamiltonian functions, we get
\begin{gather*}
\big[ \widehat{e}_{1},u_{i}\big] (H_{j}) =-\big(C_{31}^{3}+\mu _{i}C_{12}^{3}\big) \frac{\Vert v\Vert }{\alpha_{i}}. 
\end{gather*}
Therefore, our integrability condition for Hamiltonian functions becomes
\begin{gather*}
\widehat{e}_{1}\cdot \nabla \left( \frac{\Vert v\Vert }{\alpha_{i}}\right) =-\big( C_{31}^{3}+\mu _{i}C_{12}^{3}\big) \frac{\Vert v\Vert }{\alpha _{i}},
\end{gather*}
hence,
\begin{gather}
\widehat{e}_{1}\cdot \nabla \ln \left( \frac{\alpha _{i}}{\Vert v\Vert }\right) =\mu _{i}C_{12}^{3}+C_{31}^{3} \label{e43}
\end{gather}
and the proposition follows.
\end{proof}

\begin{Corollary}\label{corollary1}
The pair of Poisson structures $J_{i}=\alpha _{i}\big( \widehat{e}_{2}+\mu_{i}\widehat{e}_{3}\big) $ where $\alpha _{i}$'s are defined by~\eqref{e43} and $\mu _{i}$'s are defined by~\eqref{e20} are compatible.
\end{Corollary}

\begin{proof}
What we need is to show that (\ref{e23}) is satisf\/ied. Indeed, writing (\ref{e43}) for $\alpha _{i}$ and $\alpha_{j}$ and subtracting the second from the f\/irst, the corollary follows.
\end{proof}

Note that, for a pair of compatible Poisson structures, $J_{1}$ and $J_{2}$, the dilatation symmetry $J\rightarrow fJ$ and the additive symmetry $J_{1}+J_{2}$ do not imply that $J_{1}+fJ_{2}$ is a Poisson structure. Indeed, if we apply the Jacobi identity condition and using triple vector identity
\begin{gather*}
( J_{1}+fJ_{2}) \cdot \nabla \times (J_{1}+fJ_{2})=-\nabla f\cdot ( J_{1}\times J_{2}) =0,
\end{gather*}
which implies that
\begin{gather*}
\widehat{e}_{1}\cdot \nabla f=0.
\end{gather*}

Now we try to describe the relation between the pair of compatible Poisson structures and Hamiltonian functions. But f\/irst, we need the following lemma to describe this relation.

\begin{Lemma}\label{lemma1}
For the bi-Hamiltonian system with a pair of compatible Poisson structures defined above,
\begin{gather*}
\nabla \cdot \widehat{e}_{1}=\widehat{e}_{1}\cdot \nabla \ln \frac{\alpha_{1}\alpha _{2}( \mu _{2}-\mu _{1}) }{\Vert v\Vert ^{2}}.
\end{gather*}
\end{Lemma}

\begin{proof}
Adding the equations for integrability conditions of Hamiltonian functions (\ref{e43}) for $i=1,2$, we get
\begin{gather}
\widehat{e}_{1}\cdot \nabla \ln ( \alpha _{1}\alpha _{2}) =\widehat{e}_{1}\cdot \nabla \ln \big( \Vert v\Vert^{2}\big)+2C_{31}^{3}+(\mu _{1}+\mu _{2}) C_{12}^{3}.\label{e47}
\end{gather}
On the other hand, subtracting the equations (\ref{e20}) satisf\/ied by $\mu _{1}$ and $\mu _{2}$, and dividing by $( \mu_{2}-\mu_{1}) $,
\begin{gather}
\widehat{e}_{1}\cdot \nabla \ln (\mu _{2}-\mu _{1}) =-\big(C_{31}^{3}+C_{12}^{2}\big) -(\mu _{1}+\mu _{2}) C_{12}^{3} . \label{e48}
\end{gather}
Adding (\ref{e47}) to (\ref{e48}) and using
\begin{gather*}
\nabla \cdot \widehat{e}_{1}=C_{i1}^{i},
\end{gather*}
we get
\begin{gather*}
\widehat{e}_{1}\cdot \nabla \ln ( \alpha _{1}\alpha _{2} (\mu_{2}-\mu _{1}) ) =\widehat{e}_{1}\cdot \nabla \ln
\big(\Vert v\Vert ^{2}\big) +\nabla \cdot \widehat{e}_{1},
\end{gather*}
and the lemma follows.
\end{proof}

\begin{Proposition}\label{proposition4} Given a bi-Hamiltonian system with a pair of compatible Poisson structures, there exists a canonical pair of compatible Poisson structures $K_{1}$, $K_{2}$ with the same Hamiltonian functions $H_{1}$, $H_{2}$ such that
\begin{gather*}
K_{i}=(-1)^{i+1}\phi \nabla H_{i},
\end{gather*}
where
\begin{gather*}
\phi =\frac{\alpha _{1}\alpha _{2}(\mu _{2}-\mu _{1}) }{\Vert v\Vert }.
\end{gather*}
\end{Proposition}

\begin{proof}
Since Poisson vector f\/ields are linearly independent, one could write Hamiltonians in terms of Poisson vector f\/ields as
\begin{gather*}
\nabla H_{i}=\sigma _{i}^{j}J_{j}.
\end{gather*}
By using (\ref{e10}), we get
\begin{gather*}
\sigma _{2}^{2}=-\sigma _{1}^{1}=\frac{\Vert v\Vert }{\alpha_{1}\alpha _{2}(\mu _{2}-\mu _{1}) }.
\end{gather*}
On the other hand, we have
\begin{gather*}
\nabla \times \nabla H_{i}=\nabla \sigma _{i}^{j}\times J_{j}+ \sigma_{i}^{j}\nabla\times J_{j}=0.
\end{gather*}
Taking the dot product of both sides with $J_{1}$ and $J_{2}$, and using the compatibility condition, we obtain
\begin{gather}
\widehat{e}_{1}\cdot \nabla \ln \sigma _{j}^{i}=\frac{J_{1}\cdot (\nabla \times J_{2} ) }{\alpha _{1}\alpha _{2}( \mu_{2}-\mu_{1}) }. \label{e60}
\end{gather}
Inserting (\ref{e31}) into (\ref{e27}) and using (\ref{e60}),
\begin{gather*}
\widehat{e}_{1}\cdot \nabla \ln \sigma _{j}^{i}=-\widehat{e}_{1}\cdot \nabla \ln \phi,
\end{gather*}
which leads to
\begin{gather*}
\sigma _{j}^{i}=\frac{\Psi _{j}^{i}}{\phi },
\end{gather*}
where
\begin{gather*}
\widehat{e}_{1}\cdot \nabla \Psi _{j}^{i}=0.
\end{gather*}
Therefore, we have
\begin{gather}
\nabla H_{1} = \frac{1}{\phi }\big( \Psi _{1}^{1}J_{1}+\Psi_{1}^{2}J_{2}\big),\qquad \nabla H_{2} = \frac{1}{\phi }\big( \Psi_{2}^{1}J_{1}-\Psi_{1}^{1}J_{2}\big).\label{e65}
\end{gather}
Inserting (\ref{e65}) into (\ref{e10}), we get
\begin{gather*}
\Psi _{1}^{1}=-1,
\end{gather*}
and f\/inally,
\begin{gather*}
\nabla H_{1} = -\frac{\Vert v\Vert }{\alpha_{1}\alpha_{2}(\mu _{2}-\mu _{1}) }\big( J_{1}-\Psi_{1}^{2}J_{2}\big),\qquad
\nabla H_{2} = \frac{\Vert v\Vert }{\alpha_{1}\alpha_{2}(\mu _{2}-\mu _{1}) }\big( \Psi_{2}^{1}J_{1}+J_{2}\big).
\end{gather*}
Note that,
\begin{gather}
\nabla H_{1}\times \nabla H_{2}=-\big( 1+\Psi _{2}^{1}\Psi _{1}^{2}\big) \frac{\Vert v\Vert ^{2}}{\alpha _{1}\alpha _{2} ( \mu_{2}-\mu_{1} ) }\widehat{e}_{1}. \label{e69}
\end{gather}
For the Hamiltonians to be functionally independent, r.h.s.\ of (\ref{e69}) must not vanish, i.e.,
\begin{gather*}
1+\Psi _{2}^{1}\Psi _{1}^{2}\neq 0.
\end{gather*}
Now let us def\/ine
\begin{gather*}
K_{1} = \frac{J_{1}-\Psi _{1}^{2}J_{2}}{1+\Psi _{2}^{1}\Psi _{1}^{2}} = -\frac{\alpha _{1}\alpha _{2}(\mu _{2}-\mu _{1}) }{\big(1+\Psi _{2}^{1}\Psi _{1}^{2}\big) \Vert v\Vert }\nabla H_{1}, \qquad
K_{2} = \frac{J_{2}+\Psi _{2}^{1}J_{1}}{1+\Psi _{2}^{1}\Psi _{1}^{2}} = \frac{\alpha _{1}\alpha _{2}(\mu _{2}-\mu _{1}) }{\big(1+\Psi
_{2}^{1}\Psi _{1}^{2}\big) \Vert v\Vert }\nabla H_{2}.
\end{gather*}
By (\ref{e10}), we get
\begin{gather*} 
K_{1}\times \nabla H_{1} = K_{2}\times \nabla H_{2} = 0, \qquad
K_{2}\times \nabla H_{1} = K_{1}\times \nabla H_{2} = v.
\end{gather*}
Choosing $K_{i}$'s to be our new Poisson vector f\/ields, the proposition follows.
\end{proof}

Consequently, we can write the local existence theorem of bi-Hamiltonian systems in three dimensions.

\begin{Theorem}\label{theorem1}Any three-dimensional dynamical system
\begin{gather}
\dot{x}(t) =v ( x(t) )\label{e73}
\end{gather}
has a pair of compatible Poisson structures
\begin{gather*}
J_{i}=\alpha _{i}\big( \widehat{e}_{2}+\mu _{i}\widehat{e}_{3}\big),
\end{gather*}
in which $\mu _{i}$'s are determined by the equation
\begin{gather*}
\widehat{e}_{1}\cdot \nabla \mu _{i}=-C_{31}^{2}-\mu _{i}\big(C_{31}^{3}+C_{12}^{2}\big) -\mu _{i}^{2}C_{12}^{3},
\end{gather*}
and $\alpha _{i}$'s are determined by the equation
\begin{gather*}
\widehat{e}_{1}\cdot \nabla \ln \frac{\alpha _{i}}{\Vert v\Vert }=C_{31}^{3}+\mu _{i}C_{12}^{3}.
\end{gather*}
Furthermore, \eqref{e73} is a locally bi-Hamiltonian system with a pair of local Hamiltonian functions determined by
\begin{gather}
J_{i}= ( -1 ) ^{i+1}\phi \nabla H_{i}, \label{e77}
\end{gather}
where
\begin{gather}
\phi =\frac{\alpha _{1}\alpha _{2} ( \mu _{2}-\mu _{1} ) }{\Vert v\Vert }. \label{e78}
\end{gather}
\end{Theorem}

\section[Global existence of compatible bi-Hamiltonian structure in 3D]{Global existence of compatible bi-Hamiltonian\\ structure in 3D}\label{section3}

In this section, we investigate the conditions for which the local existence theorem holds globally. To study the global properties of the vector f\/ield $\boldsymbol{v}$ by topological means, we relate the vector f\/ield with its normal bundle. Let $E$ be the one-dimensional subbundle of $TM$ generated by $v$. Let $Q=TM/E$ be the normal bundle of $v$. By using the cross product with $\widehat{e}_{1}$, we can def\/ine a~complex structure $\Lambda$ on the f\/ibers of $Q\rightarrow M$, and $Q$ becomes a complex line bundle over~$M$.

\subsection{Bi-Hamiltonian structure in 3D with dif\/ferential forms}\label{section3.1}

In order to obtain and express the obstructions to the global existence of bi-Hamiltonian structures on orientable three manifolds by certain cohomology groups and characteristic classes, we will reformulate the problem by using dif\/ferential forms. For this purpose, let $\boldsymbol{\Omega }$ be the volume form associated to the Riemannian metric $\boldsymbol{g}$ of~$M$. Then, there is a local one-form $\boldsymbol{J}$ associated with a local Poisson bivector f\/ield~$\mathcal{J}$,
\begin{gather*}
\boldsymbol{J}=\imath _{\mathcal{J}}\boldsymbol{\Omega },
\end{gather*}
which is called the local Poisson one-form. The bi-Hamiltonian system (\ref{e10}) can be written as
\begin{gather}
\iota _{v}\boldsymbol{\Omega }=\boldsymbol{J}_{1}\wedge {\rm d}H_{2}=\boldsymbol{J}_{2}\wedge {\rm d}H_{1}. \label{e83}
\end{gather}
Note that, although the l.h.s.\ of this equality is globally def\/ined, r.h.s.\ is def\/ined only locally, therefore it holds only locally. Now the Jacobi identity is given by
\begin{gather}
\boldsymbol{J}_{i}\wedge {\rm d}\boldsymbol{J}_{i}=0\qquad \text{for} \quad i=1,2, \label{e84}
\end{gather}
and compatibility amounts to
\begin{gather*}
\boldsymbol{J}_{1}\wedge {\rm d}\boldsymbol{J}_{2}=-\boldsymbol{J}_{2}\wedge {\rm d}\boldsymbol{J}_{1}.
\end{gather*}
By (\ref{e77}), $\boldsymbol{J}_{1}$ and $\boldsymbol{J}_{2}$ can be chosen to be proportional to ${\rm d}H_{1}$ and ${\rm d}H_{2}$, respectively, and hence~(\ref{e83}) takes the form
\begin{gather*}
\iota _{v}\boldsymbol{\Omega }=\phi {\rm d}H_{1}\wedge {\rm d}H_{2}.
\end{gather*}

The Jacobi identity for Poisson 1-forms (\ref{e84}) implies the existence of 1-forms $\boldsymbol{\beta }_{i}$ such that
\begin{gather}
{\rm d}\boldsymbol{J}_{i}=\boldsymbol{\beta }_{i}\wedge \boldsymbol{J}_{i} \label{e87}
\end{gather}
for each $i=1,2$. In the next proposition we are going to show that the compatibility of Poisson structures allows us to combine $\boldsymbol{\beta }_{1} $ and $\boldsymbol{\beta }_{2}$ into a single one.

\begin{Proposition}\label{proposition5}
There is a $1$-form $\boldsymbol{\beta }$ such that
\begin{gather*}
{\rm d}\boldsymbol{J}_{i}=\boldsymbol{\beta }\wedge \boldsymbol{J}_{i} 
\end{gather*}
for each $i=1,2$.
\end{Proposition}

\begin{proof}
Applying (\ref{e87}) to the compatibility condition
\begin{gather*}
\boldsymbol{J}_{1}\wedge {\rm d}\boldsymbol{J}_{2}+\boldsymbol{J}_{2}\wedge {\rm d}\boldsymbol{J}_{1}=0,
\end{gather*}
we get
\begin{gather*}
 ( \boldsymbol{\beta }_{1}-\boldsymbol{\beta }_{2} ) \wedge \boldsymbol{J}_{1}\wedge \boldsymbol{J}_{2}=0,
\end{gather*}
which implies that
\begin{gather*}
\boldsymbol{\beta }_{1}-\boldsymbol{\beta }_{2}=b_{1}\boldsymbol{J}_{1}+b_{2}\boldsymbol{J}_{2},
\end{gather*}
and therefore, we def\/ine
\begin{gather*}
\boldsymbol{\beta }=\boldsymbol{\beta }_{1}-b_{1}\boldsymbol{J}_{1}=\boldsymbol{\beta }_{2}+b_{2}\boldsymbol{J}_{2}.
\end{gather*}
Hence
\begin{gather*}
\boldsymbol{\beta }\wedge \boldsymbol{J}_{i}=\boldsymbol{\beta }_{i}\wedge \boldsymbol{J}_{i}={\rm d}\boldsymbol{J}_{i},
\end{gather*}
and the proposition follows.
\end{proof}

Note that $\boldsymbol{\beta }$ is a $TM$-valued 1-form. Namely,
\begin{gather*}
\iota _{\widehat{e}_{1}}\boldsymbol{\beta }\neq 0
\end{gather*}
in general. Now we are going to show that by an appropriate change of Poisson forms, we may reduce it to a connection 1-form on~$Q$.

\begin{Lemma}\label{lemma2}
\begin{gather*}
\iota _{\widehat{e}_{1}}\boldsymbol{\beta }=\iota _{\widehat{e}_{1}} ( {\rm d}\ln \phi ),
\end{gather*}
where $\phi $ is the function defined in~\eqref{e78}.
\end{Lemma}

\begin{proof}
For the proof, we carry out the computation with Poisson vector f\/ields, then transform the result to dif\/ferential forms. The Jacobi identity~(\ref{e9}) implies that $\nabla \times J_{i}$ is orthogonal to~$J_{i}$ and therefore, we get
\begin{gather}
\nabla \times J_{i}=a_{i1}\widehat{e}_{1}+a_{i2}\widehat{e}_{1}\times J_{i}.\label{e96}
\end{gather}
By the def\/inition of Poisson vector f\/ields, we have
\begin{gather*}
J_{1}\times J_{2}=\phi \Vert v\Vert \widehat{e}_{1}.
\end{gather*}
We can rewrite (\ref{e96}) in the form
\begin{gather}
\nabla \times J_{i}=\frac{a_{i1}}{\phi \Vert v\Vert }J_{1}\times J_{2}+a_{i2}\widehat{e}_{1}\times J_{i}. \label{e98}
\end{gather}
Using the compatibility condition (\ref{e25}), we obtain
\begin{gather*}
a_{i1} = ( \nabla \times J_{i}) \cdot \widehat{e}_{1}, \qquad
a_{i2} = \frac{( \nabla \times J_{1}) \cdot J_{2}}{\phi\Vert v\Vert }.
\end{gather*}
Now we def\/ine
\begin{gather*}
\xi =\frac{a_{21}J_{1}-a_{11}J_{2}+( ( \nabla \times J_{1}) \cdot J_{2}) \widehat{e}_{1}}{\phi \Vert \overrightarrow{v}\Vert },
\end{gather*}
and (\ref{e98}) becomes
\begin{gather*}
\nabla \times J_{i}=\xi \times J_{i}.
\end{gather*}
After a bit of computation it is possible to show that
\begin{gather*}
\xi =\nabla \ln \phi +\widehat{e}_{1}\times \left( \frac{[ \widehat{e}_{1}\times J_{1},\widehat{e}_{1}\times J_{2}] }{\phi \Vert
v\Vert }-\widehat{e}_{1}\times \nabla \ln \Vert v\Vert\right).
\end{gather*}
Hence, we have
\begin{gather*}
\widehat{e}_{1}\cdot \xi =\widehat{e}_{1}\cdot \nabla \ln \phi
\end{gather*}
and def\/ining
\begin{gather*}
\boldsymbol{\beta }=\ast \iota _{\xi }\boldsymbol{\Omega},
\end{gather*}
the lemma follows.
\end{proof}

Now we def\/ine new Poisson 1-forms $K_{i}$
\begin{gather*}
\boldsymbol{J}_{i}=\phi \boldsymbol{K}_{i}.
\end{gather*}
Taking the exterior derivatives of both sides
\begin{gather*}
{\rm d}\boldsymbol{J}_{i}={\rm d}\phi \wedge \boldsymbol{K}_{i}+\phi {\rm d}\boldsymbol{K}_{i}=\boldsymbol{\beta }\wedge \phi \boldsymbol{K}_{i}
\end{gather*}
and dividing both sides by $\phi $,
\begin{gather*}
{\rm d}\boldsymbol{K}_{i}= ( \boldsymbol{\beta }-{\rm d}\ln \phi ) \wedge \boldsymbol{K}_{i}.
\end{gather*}
Let
\begin{gather*}
\boldsymbol{\gamma }=\boldsymbol{\beta }-{\rm d}\ln \phi.
\end{gather*}
Now, by the lemma above,
\begin{gather}
\iota _{\widehat{e}_{1}}\boldsymbol{\gamma }=\iota _{\widehat{e}_{1}}\boldsymbol{\beta }-\iota _{\widehat{e}_{1}} ( {\rm d}\ln \phi ) =0, \label{e117}
\end{gather}
therefore,
\begin{gather}
{\rm d}\boldsymbol{K}_{i}=\boldsymbol{\gamma }\wedge \boldsymbol{K}_{i}, \label{e118}
\end{gather}
where $\gamma $ is a connection on~$Q$.

\subsection[The f\/irst obstruction: the Chern class of $Q$]{The f\/irst obstruction: the Chern class of~$\boldsymbol{Q}$}\label{section3.2}

Now we try to f\/ind conditions for which a nonvanishing vector f\/ield $v$ satisf\/ies
\begin{gather}
\boldsymbol{w}=\iota _{v}\boldsymbol{\Omega }=\phi {\rm d}H_{1}\wedge {\rm d}H_{2} \label{e119}
\end{gather}
for some globally def\/ined functions $\phi $, $H_{1}$ and $H_{2}$. For a~two-form to be decomposed into the form~(\ref{e119}), f\/irst of all, the two-form must be written as a product of two globally def\/ined, linearly independent nonvanishing factors. However, such a decomposition may not exist globally. Then, the question is to decompose~$\boldsymbol{w}$ into a product of two globally def\/ined one forms $\boldsymbol{\rho }_{1}$ and $\boldsymbol{\rho }_{2}$
\begin{gather}
\boldsymbol{w}=\boldsymbol{\rho }_{1}\wedge \boldsymbol{\rho }_{2}. \label{e123}
\end{gather}
Since $v$ is a nonvanishing vector f\/ield, $\boldsymbol{w}$ is a $2$-form of constant rank $2$. If we let $S_{\boldsymbol{w}}$ to be the sub-bundle of $TM$ on which $\boldsymbol{w}$ is of maximal rank, then we have $S_{\boldsymbol{w}}\cong Q $ def\/ined above. The following theorem states the necessary and suf\/f\/icient conditions for the decomposition of a~two-form of constant rank $2s$ in the large.

\begin{Theorem}\label{theorem2} Let $\Sigma $ be an $\mathbb{R}^{n}$-bundle over a connected base space $M$. Let $\boldsymbol{w}$ be a $2$-form on $\Sigma $ of constant rank~$2s$. Let $S_{\boldsymbol{w}}$ be the subbundle of $\Sigma $ on which $\boldsymbol{w}$ is of maximal rank. $w$~decomposes if and only if
\begin{itemize}\itemsep=0pt
\item[$i)$] $S_{\boldsymbol{w}}$ is a trivial bundle.

\item[$ii)$] The representation of its normalization as a map $w_{1}\colon M\rightarrow {\rm SO}(2s)/{\rm U}(s)$ arising from any trivialization of $S_{\boldsymbol{w}}$ lifts to ${\rm SO}(2s)$~{\rm \cite{Dibag}}.
\end{itemize}
\end{Theorem}

In our case, when $s=1$, since ${\rm U}(1)\cong {\rm SO}(2)$, then ${\rm SO}(2)/{\rm U}(1)$ is a~point and it lifts to~${\rm SO}(2)$ trivially, therefore the second condition in the theorem is satisf\/ied. Hence, the necessary and suf\/f\/icient condition of decomposition is the triviality of $S_{\boldsymbol{w}}\cong Q$. Since $Q$ is a~complex line bundle, it is trivial if and only if $\boldsymbol{c}_{1}(Q)=0$, or equivalently it has a global section. Since the decomposition of the 2-form $\boldsymbol{w}$ into globally def\/ined 1-forms $\boldsymbol{\rho }_{1}$ and $\boldsymbol{\rho }_{2}$ is a necessary condition for the existence of a~global bi-Hamiltonian structure, the vanishing of the f\/irst Chern class of~$Q$ becomes a necessary condition.

However, this may not be suf\/f\/icient since the existence of a decomposition in the form (\ref{e123}) may not imply that the factors $\boldsymbol{\rho }_{i}$ satisfy
\begin{gather*}
\boldsymbol{\rho }_{i}\wedge {\rm d}\boldsymbol{\rho }_{i}=0.
\end{gather*}
In order to determine the ef\/fect of vanishing Chern class on the constructions made so far, we are going to investigate the equation~(\ref{e20}) def\/ining the Poisson one-forms. Since our Poisson one-forms and related integrability conditions are determined by the local solutions of~(\ref{e20}), they are def\/ined locally on each chart. Let $\big\{ J_{i}^{p}\big\} $ and $\big\{ J_{i}^{q}\big\} $ be the Poisson vector f\/ields in charts $( U_{p},x _{p}) $ and $ ( U_{q},x _{q}) $ around points $p\in M$ and $q\in M$, respectively. Around the point $p\in M$, the Poisson vectors $\big\{ J_{i}^{p}\big\} $ are determined by $\mu _{i}^{p},\alpha _{i}^{p}$ and the local frame $\big\{ \widehat{e}_{2}^{p}, \widehat{e}_{3}^{p}\big\} $. Given the local frame, we can write~(\ref{e20}) whose solutions are $\mu _{i}^{p}$'s, and using $\mu _{i}^{p}$'s we can determine $\alpha _{i}^{p}$'s by the equation~(\ref{e31}). Now, if $\boldsymbol{c}_{1}(Q) =0$, which is a necessary condition for the existence of global bi-Hamiltonian structure, then we have a global section of~$Q$, i.e., global vector f\/ields normal to $v$. By using the metric on $M$, normalize this global section of $Q$ and take it as $\widehat{e}_{2}$, then def\/ine $\widehat{e}_{3}=\widehat{e}_{1}\times \widehat{e}_{2}$. So we have the global orthonormal frame f\/ield $\{ \widehat{e}_{1},\widehat{e}_{2},\widehat{e}_{3}\} $. In order to understand the relation between local Poisson one-forms obtained in two dif\/ferent coordinate neighborhoods, we f\/irst need the following lemmas:

\begin{Lemma}\label{lemma3}
If two solutions $\mu _{1}(s) $ and $\mu _{2}(s) $ of the Riccati equation
\begin{gather*}
\frac{{\rm d}\mu _{i}}{{\rm d}s}=-C_{31}^{2}-\mu _{i}\big( C_{31}^{3}+C_{12}^{2}\big) -\mu _{i}^{2}C_{12}^{3}
\end{gather*}
are known, then the general solution $\mu (s) $ is given by
\begin{gather*}
\mu -\mu _{1}=K ( \mu -\mu _{2} ) e^{\int C_{12}^{3}(\mu _{2}-\mu _{1}) {\rm d}s},
\end{gather*}
where $K$ is an arbitrary constant~{\rm \cite{Ince}}.
\end{Lemma}

\begin{Lemma}\label{lemma4}
If $\boldsymbol{c}_{1}(Q) =0,$ then two pairs of compatible Poisson vector
fields $\big\{ J_{i}^{p}\big\} $ and $\big\{ J_{i}^{q}\big\} $ on $U_{p}$ and $U_{q}$ respectively, are related on $U_{p}\cap U_{q}$ by
\begin{gather*}
\frac{J_{i}^{q}}{\big\Vert J_{i}^{q}\big\Vert }=\frac{J_{i}^{p}}{\big\Vert J_{i}^{p}\big\Vert }.
\end{gather*}
\end{Lemma}

\begin{proof}Given the global frame f\/ield $\{ \widehat{e}_{2},\widehat{e}_{3}\} $ def\/ined on coordinate neighbor\-hoods~$U_{p}$ and~$U_{q}$, Riccati equations for~$\mu _{i}$'s can be written as
\begin{gather*}
\widehat{e}_{1}\cdot \nabla \mu _{i}^{r}=( \nabla \times \widehat{e}_{2}) \cdot \widehat{e}_{2}+\mu _{i}^{r}\big(( \nabla \times \widehat{e}_{2}) \cdot \widehat{e}_{3}+( \nabla \times \widehat{e}_{3}) \cdot \widehat{e}_{2}\big) +\big( \mu_{i}^{r}\big)^{2} \big(\nabla \times \widehat{e}_{3} \big) \cdot
\widehat{e}_{3}
\end{gather*}
for $r=p,q$. Therefore, on $U_{p}\cap U_{q}$, $\mu _{i}^{p}$ and $\mu_{i}^{q} $ are four solutions of the same Riccati equation for $i=1,2$. By the lemma above we have
\begin{gather}
\mu _{i}^{q}-\mu _{1}^{p}=K_{i}^{pq}\big( \mu _{i}^{q}-\mu _{2}^{p}\big) e^{\int C_{12}^{3}\big( \mu _{2}^{p}-\mu _{1}^{p}\big) {\rm d}s}. \label{e130}
\end{gather}
Now, using the compatibility condition (\ref{e23}),
\begin{gather*}
C_{12}^{3}\big( \mu _{2}^{p}-\mu _{1}^{p}\big) =\widehat{e}_{1}\cdot\nabla \ln \frac{\alpha _{2}^{p}}{\alpha _{1}^{p}},
\end{gather*}
(\ref{e130}) becomes
\begin{gather}
\mu _{i}^{q}-\mu _{1}^{p}=K_{i}^{pq}\big( \mu _{i}^{q}-\mu _{2}^{p}\big) \frac{\alpha _{2}^{p}}{\alpha _{1}^{p}}, \label{e132}
\end{gather}
where
\begin{gather}
\widehat{e}_{1}\cdot \nabla K_{i}^{pq}=0. \label{e133}
\end{gather}
Multiplying both sides by $\alpha _{1}^{p}\alpha _{i}^{q}$ in (\ref{e132}), gives
\begin{gather}
J_{i}^{q}\times J_{1}^{p}=K_{i}^{pq}J_{i}^{q}\times J_{2}^{p}. \label{e135}
\end{gather}
Rearranging (\ref{e135}), we obtain
\begin{gather*}
J_{i}^{q}\times \big( J_{1}^{p}-K_{i}^{pq}J_{2}^{p}\big) =0.
\end{gather*}
Using (\ref{e133}) and the compatibility, we can take
\begin{gather*}
\widetilde{J}_{i}^{p}=J_{1}^{p}-K_{i}^{pq}J_{2}^{p}
\end{gather*}
to be our new Poisson vector f\/ields on the neighborhood $U_{p}$, and obtain
\begin{gather*}
J_{i}^{q}\times \widetilde{J}_{i}^{p}=0.
\end{gather*}
By compatibility these new Poisson vector f\/ields $\widetilde{J}_{i}^{p}$ produce functionally dependent Hamilto\-nians and therefore, for the simplicity of notation, we will assume without restriction of generality that
\begin{gather*}
\widetilde{J}_{i}^{p}=J_{i}^{p}
\end{gather*}
and the lemma follows.
\end{proof}

Then, we have the following result:

\begin{Theorem}\label{theorem3}
There exist two linearly independent global sections $\widehat{j}_{i}$ of $Q$ satisfying
\begin{gather}
\widehat{j}_{i}\cdot \big( \nabla \times \widehat{j}_{i}\big) =0\label{e140}
\end{gather}
if and only if $\boldsymbol{c}_{1}(Q) =0$.
\end{Theorem}

\begin{proof}
The forward part is trivial since the existence of a global section of the complex line bundle $Q$ implies that $Q$ is trivial, and hence $\boldsymbol{c}_{1}(Q) $ vanishes. For the converse, we def\/ine
\begin{gather*}
\widehat{j}_{i}^{p}=\frac{J_{i}^{p}}{\big\Vert J_{i}^{p}\big\Vert }
\end{gather*}
and the lemma implies that $j_{i}^{p}=j_{i}^{q}$ on $U_{p}\cap U_{q}$ and the theorem follows.
\end{proof}

The lemma above states the reason why one may fail to extend a local pair of compatible Poisson vector f\/ields into a global one, even if $\boldsymbol{c}_{1}(Q) =0$. In order to do so one should have $J_{i}^{q}=J_{i}^{p}$ on $U_{p}\cap U_{q}$. However, not the Poisson vector f\/ields but their unit vector f\/ields can be globalized. Since
\begin{gather*}
\widehat{e}_{1}\cdot \nabla \frac{\big\Vert J_{2}^{p}\big\Vert }{\big\Vert J_{1}^{p}\big\Vert }\neq 0 
\end{gather*}
in general, they may not lead to a pair of compatible Poisson structures. Now we take $\widehat{j}_{1}$ as our f\/irst global Poisson vector f\/ield, and check whether we can f\/ind another global Poisson vector f\/ield compatible with this one by rescaling~$\widehat{j}_{2}$.

\subsection{Second obstruction: Bott class of the complex codimension 1 foliation}\label{section3.3}

Since $v$ is a nonvanishing vector f\/ield on $M$, it def\/ines a real codimension two foliation on~$M$ by orbits of~$v$. Since $Q=TM/E$ is a~complex line bundle on $M$, this foliation has complex codimension one. Now, by assuming our primary obstruction which is the vanishing of the Chern class, we compute the Bott class of the complex codimension one foliation as def\/ined in~\cite{Bott}, which is studied in detail in~\cite{Taro}, and then show that the system admits two globally def\/ined compatible Poisson structures if and only if the Bott Class is trivial.

For the rest of our work, we will assume that $Q$ and its dual $Q^{\ast }$ are trivial bundles. By~(\ref{e140}), $Q^{\ast }$~has two global sections $\widehat{\boldsymbol{j}}_{i} = ( ^{\ast }\imath _{\widehat{j}_{i}}\boldsymbol{\Omega } )$ satisfying
\begin{gather}
{\rm d}\widehat{\boldsymbol{j}}_{i}=\boldsymbol{\Gamma }_{i}\wedge \widehat{\boldsymbol{j}}_{i}\label{e143}
\end{gather}
for globally def\/ined $\boldsymbol{\Gamma }_{i}$'s. These $\widehat{\boldsymbol{j}}_{i}$'s are related with the local Poisson one-forms $\boldsymbol{J}_{i}^{p}$ by
\begin{gather}
\boldsymbol{J}_{i}^{p}=\big\Vert \boldsymbol{J}_{i}^{p}\big\Vert \widehat{\boldsymbol{j}}_{i}. \label{e144}
\end{gather}
By (\ref{e118}), we have
\begin{gather}
{\rm d}\boldsymbol{J}_{i}^{p}=\boldsymbol{\gamma }^{p}\wedge \boldsymbol{J}_{i}^{p}.\label{e145}
\end{gather}
Inserting (\ref{e144}) and (\ref{e145}) into (\ref{e143}), we also have
\begin{gather}
{\rm d}\widehat{\boldsymbol{j}}_{i}=\big( \boldsymbol{\gamma }^{p}-{\rm d}\ln \big\Vert \boldsymbol{J}_{i}^{p}\big\Vert \big) \wedge \widehat{\boldsymbol{j}}_{i}. \label{e146}
\end{gather}
Redef\/ining $\boldsymbol{\Gamma} _{i}$'s if necessary, comparing (\ref{e143}) with (\ref{e146}), we get
\begin{gather}
\boldsymbol{\Gamma }_{i}=\boldsymbol{\gamma }^{p}-{\rm d}\ln \big\Vert \boldsymbol{J}_{i}^{p}\big\Vert. \label{e147}
\end{gather}

\begin{Proposition}\label{proposition6}
Let $\boldsymbol{\kappa }$ be the curvature two-form of $Q$. There exists a~compatible pair of global Poisson structures if and only if
\begin{gather*}
\boldsymbol{\Xi }= ( \boldsymbol{\Gamma }_{1}-\boldsymbol{\Gamma }_{2} ) \wedge \boldsymbol{\kappa } \label{e148}
\end{gather*}
is exact.
\end{Proposition}

\begin{proof}
Since $\widehat{\boldsymbol{j}}_{1}$ and $\widehat{\boldsymbol{j}}_{2}$ may not be compatible, we introduce a local Poisson form $\boldsymbol{j}^{p}$ def\/ined on the coordinate neighborhood $U_{p}$ of $p\in M$, which is compatible with $\widehat{\boldsymbol{j}}_{1}$ and parallel to $\widehat{\boldsymbol{j}}_{2}$ i.e.,
\begin{gather}
\boldsymbol{j}^{p}=f^{p}\widehat{\boldsymbol{j}}_{2} \label{e149}
\end{gather}
and
\begin{gather}
\widehat{\boldsymbol{j}}_{1}\wedge {\rm d}\boldsymbol{j}^{p}+\boldsymbol{j}^{p}\wedge {\rm d}\widehat{\boldsymbol{j}}_{1}=0. \label{e150}
\end{gather}
Now (\ref{e149}) implies that
\begin{gather}
{\rm d}\boldsymbol{j}^{p}=\big( \boldsymbol{\Gamma }_{2}+{\rm d}\ln f^{p}\big) \wedge \boldsymbol{j}^{p}. \label{e151}
\end{gather}
Putting (\ref{e143}) and (\ref{e151}) into (\ref{e150}) and using (\ref{e149}), we get
\begin{gather*}
\big( \boldsymbol{\Gamma }_{1}-\boldsymbol{\Gamma }_{2}-{\rm d}\ln f^{p}\big) \wedge \widehat{\boldsymbol{j}}_{1}\wedge \boldsymbol{j}^{p}=0
\end{gather*}
which implies
\begin{gather}
 ( \boldsymbol{\Gamma }_{1}-\boldsymbol{\Gamma }_{2} ) \wedge \widehat{\boldsymbol{j}}_{1}\wedge \widehat{\boldsymbol{j}}_{2}={\rm d}\ln f^{p}\wedge \widehat{\boldsymbol{j}}_{1}\wedge \widehat{\boldsymbol{j}}_{2}. \label{e153}
\end{gather}
Our aim here is to f\/ind the obstruction to extending $f^{p}$ to $M$, or for (\ref{e153}) to hold globally. For this purpose, we consider the connections on~$Q$ def\/ined by~$\Gamma _{i}$'s. By (\ref{e147}), we def\/ine the curvature of these connections to be
\begin{gather*}
\boldsymbol{\kappa }={\rm d}\boldsymbol{\Gamma }_{i}={\rm d}\boldsymbol{\gamma }^{p}.
\end{gather*}
Taking the exterior derivative of (\ref{e145}) and using (\ref{e144}), we get
\begin{gather*}
{\rm d}\boldsymbol{\gamma }^{p}\wedge \boldsymbol{J}_{i}^{p}={\rm d}\boldsymbol{\gamma }^{p}\wedge \widehat{\boldsymbol{j}}_{i}=0,
\end{gather*}
which leads to
\begin{gather}
\kappa ={\rm d}\boldsymbol{\gamma }^{p}=\varphi \widehat{\boldsymbol{j}}_{1}\wedge \widehat{\boldsymbol{j}}_{2}. \label{e156}
\end{gather}
Now multiplying both sides of (\ref{e153}) with $\varphi $,
\begin{gather*}
( \boldsymbol{\Gamma }_{1}-\boldsymbol{\Gamma }_{2}) \wedge \boldsymbol{\kappa }={\rm d}\ln f^{p}\wedge \boldsymbol{\kappa }={\rm d}\big( \big( \ln f^{p}\big) \boldsymbol{\kappa }\big)
\end{gather*}
and the proposition follows.
\end{proof}

Now we are going to show that the cohomology class of $\boldsymbol{\Xi }$ vanishes if and only if the Bott class of the complex codimension~1 foliation vanishes. Since $Q$ is a complex line bundle we have
\begin{gather*}
\boldsymbol{c}_{1}(Q) = [ \boldsymbol{\kappa } ]
\end{gather*}
and the vanishing of $\boldsymbol{c}_{1}(Q) $ is a necessary condition
\begin{gather*}
\boldsymbol{c}_{1}={\rm d}\boldsymbol{h}_{1}.
\end{gather*}
So we have
\begin{gather*}
\boldsymbol{c}_{1}= [ \boldsymbol{\kappa } ] =\big[ {\rm d}\boldsymbol{\gamma }^{p}\big],
\end{gather*}
which implies that on $U_{p}$
\begin{gather*}
\boldsymbol{h}_{1}=\boldsymbol{\gamma }^{p}+{\rm d}\ln h^{p}.
\end{gather*}
Then, the Bott class \cite{Bott} becomes
\begin{gather*}
\boldsymbol{h}_{1}\wedge \boldsymbol{c}_{1}=\big( \boldsymbol{\gamma }^{p}+{\rm d}\ln h^{p}\big) \wedge {\rm d}\boldsymbol{\gamma }^{p}={\rm d}\ln h^{p}\wedge \boldsymbol{\kappa }+\boldsymbol{\gamma }^{p}\wedge {\rm d}\boldsymbol{\gamma }^{p}.
\end{gather*}
Now by (\ref{e117}) and (\ref{e156}) we have
\begin{gather*}
\boldsymbol{\gamma }^{p}\wedge {\rm d}\boldsymbol{\gamma }^{p}=0,
\end{gather*}
and therefore,
\begin{gather*}
\boldsymbol{h}_{1}\wedge \boldsymbol{c}_{1}={\rm d}\big( \big( \ln h^{p}\big) \kappa\big).
\end{gather*}
Since $\boldsymbol{h}_{1}$ is globally def\/ined, on $U_{p}\cap U_{q}$ we have
\begin{gather*}
\boldsymbol{h}_{1}=\boldsymbol{\gamma }^{p}+{\rm d}\ln h^{p}=\boldsymbol{\gamma }^{q}+{\rm d}\ln h^{q}
\end{gather*}
and
\begin{gather}
\boldsymbol{\gamma }^{p}-\boldsymbol{\gamma }^{q}={\rm d}\ln \frac{h^{q}}{h^{p}}.
\label{e166}
\end{gather}
Now we have the following theorem:

\begin{Theorem}\label{theorem4}
The cohomology class of $\boldsymbol{\Xi }$ vanishes if and only if the Bott class of the complex codimension one foliation defined by the nonvanishing vector field vanishes.
\end{Theorem}

\begin{proof}
If the Bott class vanishes, then we have a globally def\/ined function~$h$ such that
\begin{gather*}
{\rm d} ( ( \ln h ) \boldsymbol{\kappa } ) =0.
\end{gather*}
Then, choosing $f=h$ leads to a compatible pair of global Poisson structures. Conversely, if there is a pair of globally def\/ined compatible Poisson structures, then $\boldsymbol{\gamma }$ becomes a global form, and by~(\ref{e166}) we have
\begin{gather*}
{\rm d}\ln \frac{h^{q}}{h^{p}}=0
\end{gather*}
on $U_{p}\cap U_{q}$. Therefore,
\begin{gather*}
\ln h^{q}-\ln h^{p}=c^{qp},
\end{gather*}
where $c^{qp}$ is a constant on $U_{p}\cap U_{q}$. Now, f\/ixing a point $x_{0}\in U_{p}\cap U_{q}$
\begin{gather*}
c^{qp}=\ln h^{q} ( x_{0} ) -\ln h^{p} ( x_{0} ) =\ln c^{q}-\ln c^{p},
\end{gather*}
we obtain
\begin{gather*}
\frac{h^{p}}{c^{p}}=\frac{h^{q}}{c^{q}}=h,
\end{gather*}
where $h$ is a globally def\/ined function, and
\begin{gather*}
{\rm d}\ln h={\rm d}\ln h^{p}.
\end{gather*}
Therefore,
\begin{gather*}
 [ \boldsymbol{h}_{1}\wedge \boldsymbol{c}_{1} ] = [ {\rm d} (( \ln h) \boldsymbol{\kappa })] =0 \label{e173}
\end{gather*}
and the theorem follows.
\end{proof}

\subsection*{Acknowledgements}
We are indebted to Professor Turgut \"{O}nder for his help during this work. We also thank to the anonymous referees for their comments and corrections.

\pdfbookmark[1]{References}{ref}
\LastPageEnding

\end{document}